\documentclass[10pt]{article}
\usepackage[tbtags]{amsmath}
\usepackage{amsthm,amssymb,color}

\allowdisplaybreaks[4]
\definecolor{c20}{rgb}{0.,0,0.}
\definecolor{c30}{rgb}{0.,0.,0}
\definecolor{c40}{rgb}{0,00,00}
\definecolor{c50}{rgb}{0,0,0}

\def\EH#1{\textcolor{c30}{#1}}
\def\EH#1{\textcolor{c30}{#1}}

\def\pE#1{\textcolor{c30}{#1}}
\def\rE#1{\textcolor{c20}{#1}}
\def\rE#1{#1}
\def\EH#1{#1}

\def\qE#1{\textcolor{c30}{#1}}
\def\qE#1{#1}
\def\pE#1{#1}

\def\uE#1{\textcolor{c30}{#1}}
\def\vE#1{\textcolor{c30}{#1}}

\def\aH#1{\textcolor{c50}{#1}}
\def\cH#1{\textcolor{c30}{#1}}
\def\cH#1{#1}
\def\vE#1{#1}
\def\aH#1{\textcolor{c30}{#1}}
\def\aH#1{#1}
\def\bH#1{\textcolor{c30}{#1}}
\def\bH#1{#1}

\def\hH#1{\textcolor{c30}{#1}}
\def\hH#1{#1}

\def\cK#1{\textcolor{c40}{#1}}

\def\aK#1{\textcolor{c40}{#1}}
\def\aK#1{#1}

\def\cK#1{\textcolor{green}{#1}}
\def\cK#1{#1}

\def\uE#1{ #1 }

\def\ftj{\cE{h}}
\def\luj{\left(\frac u{\lambda_j}\right)}

\def\luiu{\left(\frac u{\lambda}\right)}

\def\verDraft{1}
\def\verFinal{2}
\def\ver{2}
\ifx\ver\verDraft
\newcommand{\notindraft}[1]{}
\else
\newcommand{\notindraft}[1]{#1}
\fi

\ifx\ver\verDraft
\newcommand{\notinfinal}[1]{#1}
\fi
\ifx\ver\verFinal
\newcommand{\notinfinal}[1]{}
\fi

\def\ver{1}
\ifx\ver\verDraft
\definecolor{c20}{rgb}{0.,0.7,0.}
\definecolor{c30}{rgb}{0.,0.,1.}
\definecolor{c40}{rgb}{1,0.1,0.7}
\definecolor{c50}{rgb}{1,0,0}
\else
\definecolor{c20}{rgb}{0.,0,0.}
\definecolor{c30}{rgb}{0.,0.,0}
\definecolor{c40}{rgb}{0,00,00}
\definecolor{c50}{rgb}{0,0,0}

\fi
\def\cD#1{\textcolor{c20}{#1}}
\def\cE#1{\textcolor{c50}{#1}}

\def\cD#1{#1}
\def\cE#1{#1}

\newcommand{\PP}{\mathbb P}
\newcommand{\mean}[1]{\mathbb E\left[ #1\right]}

\newtheorem{theo}{Theorem}[section]
\newtheorem{sat}[theo]{Proposition}
\newtheorem{de}[theo]{Definition}
\newtheorem{lem}[theo]{Lemma}
\newtheorem{exxa}[theo]{Example}

\newtheorem{korr}[theo]{Corollary}
\newtheorem{remark}[theo]{Remark}
\newtheorem{remarks}[theo]{Remarks}

\newcommand{\nelem}[1]{{Lemma \ref{#1}}}

\newcommand{\netheo}[1]{{Theorem \ref{#1}}}

\newcommand{\prooftheo}[1]{ \textsc{Proof of Theorem} \ref{#1} }

\newcommand{\kb}[1]{\boldsymbol{#1}}
\newcommand{\vk}[1]{\kb{#1}}

\newcommand{\ve}{\varepsilon}
\def\fracl#1#2{\biggr(\frac{#1}{#2} \biggl) }

\newcommand{\pk}[1]{\mbox{\rm$\mathbb{P}$} \{#1\} }

\newcommand{\R}{\mathbb{R}}

\newcommand{\inr}{\in \R}

\newcommand{\ldot}{,\ldots,}

\newcommand{\limit}[1]{\lim_{#1 \to   \infty}}

\newcommand{\BQN}{\begin{eqnarray}}
\newcommand{\EQN}{\end{eqnarray}}
\newcommand{\BQNY}{\begin{eqnarray*}}
\newcommand{\EQNY}{\end{eqnarray*}}

\newcommand{\BS}{\begin{sat}}
\newcommand{\ES}{\end{sat}}
\newcommand{\BT}{\begin{theo}}
\newcommand{\ET}{\end{theo}}
\newcommand{\BK}{\begin{korr}}
\newcommand{\EK}{\end{korr}}

\newcommand{\BD}{\begin{de}}
\newcommand{\ED}{\end{de}}

\newcommand{\BIT}{\begin{itemize}}
\newcommand{\EIT}{\end{itemize}}
\newcommand{\BDI}{\begin{description}}
\newcommand{\EDI}{\end{description}}

\newcommand{\BRM}{\begin{remarks}}
\newcommand{\ERM}{\end{remarks}}

\newcommand{\QED}{\hfill $\Box$}

\newcommand{\IF}{\infty}
\newcommand{\BTH}{\begin{theo}}
\newcommand{\ETH}{\end{theo}}
\newcommand{\BPR}{\begin{sat}}
\newcommand{\EPR}{\end{sat}}

\newcommand{\BEX}{\begin{exxa}}
\newcommand{\EEX}{\end{exxa}}

\newcommand{\BC}{\begin{cases}}
\newcommand{\EC}{\end{cases}}
\newcommand{\COM}[1]{}
\newcommand{\BL}{\begin{lem}}
\newcommand{\EL}{\end{lem}}

\def\SI{\Sigma}

\def\U{\vk{U}}

\def\1d{\{1 \ldot d\}}

\def\Zju{\cE{X_j(u)}}

\def\SSU{\cE{S(u)}}%\sum_{i=1}^d \Ziu}

\def\YYu{\vk{X}(u)}

\def\ftj{\cE{h}}

\def\ee{\Upsilon(u)}

\begin{document}

\begin{center}
\thispagestyle{empty}

{\Large \bf Second Order Asymptotics of Aggregated Log-Elliptical Risk}

       \vskip 1.4 cm

         \centerline{\large Dominik Kortschak$^{a}$\footnote{E-mail: kortschakdominik@gmail.com\\ DK was  supported by the
the MIRACCLE-GICC project and the Chaire d'excellence ``Generali -- Actuariat responsable: gestion des risques naturels et changements climatiques.``} and Enkelejd Hashorva$^{b}$\footnote{E-mail: Enkelejd.Hashorva@unil.ch\\ EH kindly acknowledges partial support by  Swiss National Science Foundation Grants 200021-134785 and 200021-1401633/1 and by RARE -318984, a Marie Curie International Research Staff Exchange Scheme Fellowship within the 7th European Community Framework Programme.}}
\vspace{2 cm}
$~^a$
Universit\'e de Lyon, F-69622, Lyon, France; Universit\'e Lyon 1, Laboratoire SAF, EA 2429, Institut de Science Financi\`ere et d'Assurances, 50 Avenue Tony Garnier, F-69007 Lyon, France\\
$~^b$ Department of Actuarial Science, Faculty of Business and Economics, University of Lausanne, UNIL-Dorigny 1015 Lausanne, Switzerland\\

       \vskip 0.4 cm

%  \today{}

\end{center}
%\newpage

%%%%%%%%%%%%%%%%%%%%%%%%%%%%%%%%%%%%%%%%%%%%%%
{\bf Abstract:} In this paper we establish the error rate of first order asymptotic approximation for the tail probability of sums of
log-elliptic\qE{al} risks. Our approach is motivated by extreme value theory which allows us to impose only some \qE{weak}
 asymptotic conditions \rE{satisfied} in particular \qE{by} log-normal risks. Given the wide range of applications of \rE{the} log-normal model in finance and insurance
our result is of interest for both rare-event simulations and numerical calculations.
%Our model is very flexible with \qE{multivarite} dependence structure \qE{parametrized} by the threshold, which is \qE{important} for rare-event %simulations.
 \qE{We present numerical examples which illustrate that the second order approximation derived in this paper significantly
improves over the first order approximation.}\\


%%%%%%%&%%%&%%%%%%%%%%%%%%%%%%%%&&%%%%%%%%%%%%%%%%%%%%%%%%%%%%%%&%%%%%%%%%%%%%%%%%%%%%%
\section{Introduction}
Modeling \qE{multivariate} dependent risks is an important task of actuaries involved in risk management, pricing and loss reserving. The standard and most common
model used in practice, both in insurance and finance, is that of dependent log-normal risks, see e.g., Mitra and Resnick (2009),
Foss and Richards (2010) or  Asmussen et al.\ (2011).  Despite the tractability of multivariate log-normal distribution, the first result which
derives the asymptotic tail behaviour of \rE{the sum of} log-normal risks \rE{appeared recently} in Asmussen and Rojas-Nandayapa (2008), see
 also Albrecher et al.\ (2006).  The recent contribution Asmussen et al.\ (2011) derives an explicit asymptotic expansion ($u\to \IF$) of
$$\pk{S(u)> u}, \quad \text{ \hH{with  } }\cH{S(u)= X_1(u)+ \cdots + X_d(u)},$$
where  $X_i(u), i\le d$ is a $d$-dimensional log-normal random vector with underlying covariance matrix depending on the \qE{threshold}  $u$.
In the aforementioned paper the consideration of parametrized \rE{risks} \qE{leads} to the introduction of novel importance sampling estimators of $\pk{S(u)> u}$.\\
Given the fact that Normal random vectors are a canonical example of the elliptically symmetric ones, it is natural to model the aggregated risk utilizing a log-elliptical framework,
 which has been recently discussed in Rojas-Nandayapa (2008), Kortschak and Hashorva (2013)  and Hashorva (2013). The latter two papers derived (under different conditions) the following asymptotic expansion
\BQN\label{MT}
\PP\left(\SSU >u\right)%:=\PP\left(\sum_{i=1}^d X_i(u)>u\right)
\sim \sum_{i=1}^d \PP\left(X_i(u)>u\right), \quad u\to \IF,
\EQN
with  $X_i(u), i\le d$ the components of some $d$-dimensional log-elliptical random vector indexed by $u$ \pE{(here} $a(u) \sim b(u)$
stands for the asymptotic equivalence as $u\to \IF$ of two functions $a(\cdot),b(\cdot)$\pE{)}.

The \pE{principal} goal of this contribution is
the precise quantification of the error of the approximation claimed in \eqref{MT}. Specifically for $\ee$ defined as
$$
\ee := \PP\left(\SSU >u\right)- \sum_{i=1}^d \PP\left(X_i(u)>u\right), \quad u\to \IF
$$
we derive in the main result (\netheo{theorem:secondorder} below) the rate of convergence of $\ee$ to 0 as  $u\to \IF$.
As was already observed in Mitra and Resnick (2009) the first order approximation \rE{in (1.1)} for positive\qE{ly} correlated random variables can be rather crude even in the \cE{bivariate case $d=2$}. Now the obvious motivation of our paper is to improve this rather crude first order approximation. The numerical examples in Section \ref{section:numerics} show that the second order approximation significantly improves over the first order \rE{one}.\\

Two essential properties of log-normal risks are \rE{crucial} for the derivation of the tail asymptotic expansion in \eqref{MT}: a) the \rE{univariate} log-normal distribution belongs to the Gumbel max-domain of attraction (see below for definition),
and b) log-normal risks and in particular \qE{Normal} \rE{ones} are asymptotically independent, see e.g., Resnick (1987). \\
The derivation of \eqref{MT} for log-elliptical risks is strongly based on assumptions which agree with a) and b) above. In order to derive the \EH{asymptotics} of $\ee$, we shall impose some additional \rE{restrictions} on
the probability density function of log-elliptical risks. Our result is new even
when $(\log X_1(u) \ldot  \log X_d(u))$ is a $d$-dimensional Normal random vector with mean zero and non-singular covariance matrix $\Sigma$.
% which has all entries in the main diagonal equal 1.
For this case assuming  for simplicity that the off-diagonal elements of $\SI$ are equal to $\rho \in (-1,1)$, we \rE{obtain}
\BQN\label{eee}
\ee &\sim &  \frac{d(d-1) \exp( (1- \rho^2)/2)}{ \sqrt{2 \pi}  u^{1- \rho}} \exp(- (\log u)^2/2), \quad u\to \IF.
\EQN
\vE{The speed of convergence \pE{of} $\ee$ to 0 is shown by \eqref{eee} to decrease with $\rho$ increasing, i.e., the more dependence the worse the approximation. When $\rho=0$, then \eqref{eee} shows that
$$ \ee \sim d(d-1) \exp( 1/2)f_*(u)$$
as $u\to \IF$ with $f_*$ the probability density function of $X_1(u)$. As expected, \eqref{eee} \rE{implies} further that when \rE{the dimension} $d$ increases the quality of approximation \rE{also} decreases.}

Organisation of the rest of the paper:
In \netheo{theorem:secondorder} below we present our main result. Numerical comparisons are given in  Section \ref{section:numerics}.
Several auxiliary results and the proof of the main result are \qE{displayed} in Section \ref{section:proofs}.

\section{Results}\label{sec:main}
Let  $R$ be a positive random variable with distribution function $H$ being independent
of $\U$ which is uniformly distributed on the unit sphere of $\R^d$ (with respect to the $L_2$-norm). For $A_u,u>0$ a sequence of $d\times d$
non-singular matrices we \rE{shall set} $\Sigma_u= A_u A_u^\top$.  In the sequel we suppose that \rE{the elements of $\Sigma_u$ satisfy}
\BQN\label{SIu}
\sigma_{11}(u)=\cdots=\sigma_{dd}(u)=1, \quad \sigma_{ij}(u)\in [-1,1], \quad i\not=j,\quad u>0
\EQN
and \rE{further} $\beta_i, \lambda_i, i\le d$ are positive constants \rE{such that}
\BQN\label{betaC}
 0 <  \beta_d \le \cdots \le \beta_1 < \IF, \quad \lambda_1=\max_{\beta_i=\beta_1} \lambda_i.
 \EQN
For given $\gamma_u,u>0$ positive constants \rE{satisfying}  $\limit{u} \gamma_u=\gamma\in (0,\IF)$ we define a $d$-dimensional random vector
$$\YYu:=(\lambda_1 Z_1(u)^{\beta_1 \cE{\gamma_u} },\cdots, \lambda_d Z_{\cE{d}}(u)^{\beta_d \gamma_u})^\top,$$
where
$$ (Z_1(u),\ldots,Z_d(u))^\top= \exp( R A_u \U), \quad u>0.$$
The class of log-elliptical risks considered in this paper are such that $R$ has distribution function $F$ with infinite upper endpoint \qE{satisfying}
\BQN \label{eq:rdfd:2}
\limit{u}
\frac{1- F(u+xb(u))} {1- F(u)} &=& \exp(-x),\quad \forall x\inr,
\EQN
\qE{with some positive scaling function $b(\cdot)$.}  When \eqref{eq:rdfd:2} holds,  we say that  $R$ \rE{(and alternatively also $F$)} is  in the
\rE{max-domain of attraction (MDA)} of the  Gumbel distribution $\Lambda(x)=\exp(-\exp(-x)), x\inr$
with positive scaling function $b(\cdot)$. If we assume further that
\BQN \label{eq:main:e}
\limit{u} b(u)&=& 0,
\EQN
then for $j\le d$
\def\luj{\left(\frac u{\lambda_j}\right)}
\BQN
\lim_{u\to\infty}\frac{\PP\left(X_j(u)>u+xe^*_j(u)\right)}{\PP\left(X_j(u)>u\right)}=\exp(-x), \quad x\inr,
\EQN
where
 \BQN\label{eq:eyj}
 e_j^*(u)&:=&  \beta_j\gamma_u u e\left( \luj^ {\frac 1 {\beta_j\gamma_u}}\right)\luj^ {-\frac 1 {\beta_j\gamma_u}},
\EQN
with
\BQN\label{eub}
 e(u)&= &u b(\log u), \quad u>0.
\EQN
Hereafter we shall assume that
\BQN\label{euINF}
\lim_{u\to\infty} e(u)&=&\infty.
\EQN
The following theorem presented in Kortschak and Hashorva (2013) establishes the first order asymptotics of aggregated log-elliptical \rE{risk}.
\BT
\label{mytheo3}   Suppose that \eqref{eq:main:e} and \EH{\eqref{euINF} hold} and for  $j$  with $\beta_j=\beta_1$  and every $\epsilon>0$, $c>0$ there exists some  $u_0$
such that for all $u>u_0$
\BQN\label{conditionrho}
\sigma_{\aH{ij}}(u)+c \sqrt{\frac{1-\sigma_{\aH{ij}}(u)^2} {\log(u)}}
&\le & \cH{\frac{\beta_{\aH{j}}}{\beta_i}} \frac {\log(\epsilon e^*_{\cH{i}}(u))}{\log(u)}
\EQN
holds for all $i\not=j$, then \eqref{MT} is satisfied.
\ET

{\bf Remark}:  a)
If $\beta_i=\beta_j$ and $\lim_{u\to\infty} \log(e_j^*(u))/\log(u)=1$, then condition \eqref{conditionrho} is for example fulfilled when
\BQN\label{remark:conditionrho}
\limsup_{ u\to \infty} \frac{-\log\left(\frac{e^*_j(u)}{u}\right) }{(1-\sigma_{ij}(u))\log(u)} <1.
\EQN
\qE{b) If $\sigma_{ij}(u) < \kappa < \lim_{u\to\infty} (e_i^*(u))/u$  for all large $ u$ and $i\not =j$, then condition \eqref{conditionrho} holds}.

\def\wT{\widetilde{R}}
In order to derive the asymptotics of the error term $\Upsilon(u)$ we shall impose \qE{some} additional  restrictions.
Both conditions \eqref{eq:rdfd:2} and \eqref{eq:main:e} imply that the  distribution function
$\widetilde F$ of $\widetilde R:= \exp(R)$ is in Gumbel MDA with scaling function \EH{$e(\cdot)$ defined in \eqref{eub}.}  Below, that assumption will be strengthened to $\widetilde{F}$ is eventually differentiable with continuous \rE{probability density function} $\widetilde{f}$ such that
the following von Mises condition
\BQN\label{eq:fe}
\lim_{u\to\infty}\frac{\widetilde{f}(u+x e(u))}{\widetilde{f}(u)}=\exp(-x) \quad
\EQN
holds for all $x\inr$. We formulate next our \rE{principal} result.% under further restrictions on the scaling function $e(u)$.

\BT \label{theorem:secondorder}
Under the assumptions of Theorem \ref{mytheo3}, suppose further that the scaling function $e(\cdot)$ is \vE{ultimately} monotone increasing and satisfies
\BQN\label{coEU}
\lim_{\lambda\to1} \limsup_{u\to\infty} \frac{e(\lambda u)} {e(u)}&=&1.
\EQN
If \eqref{eq:fe} holds, and  further %for  \cE{some $c_{\cD{i}} \in [0,\infty)$ }
\BQN\label{condc}
\lim_{u\to\infty} \frac{\log(u)e^*_i(u)}{u} =c_i\in [0,\IF), \quad 1 \le i \le d,
\EQN
then we have
\BQN\label{secondorder}
\Upsilon(u)&
\sim &\sum_{j=1}^d \left(\sum_{i\not=j}\frac{\lambda_i}{\cD{\beta_j\gamma}}
\exp\left( \frac{ c_{\cD{j}}\left( 1-\sigma_{ij}^2( \cE{u})\right)} {  2 \cE{ (\beta_j/\beta_i)}^{2 } }\right) \left(\frac{u}{\lambda_j}\right)^{\frac{\beta_i \sigma_{ij}(u)}{ \beta_j}} \right)
\frac{1}{e^*_j(u)}
\PP\left( \Zju>u\right), \quad  u\to \IF.
\EQN
\ET

\begin{remark} a)
For any $i\not=j, i,j\le d$ after some long calculations %some tedious calculations that we leave to the reader show that
\begin{align*}
\mean {\lambda_i Z_i(u)^{\beta_i \gamma(u)} \bigl \lvert \lambda_j Z_j(u)^{\beta_j \gamma(u)} =u}
\sim
\frac{\lambda_i}{\cD{\beta_j\gamma}}
\exp\left( \frac{ c_{\cD{j}}\left( 1-\sigma_{ij}^2( \cE{u})\right)} {  2 \cE{ (\beta_j/\beta_i)}^{2 } }\right)\left(\frac{u}{\lambda_j}\right)^{\frac{\beta_i \sigma_{ij}(u)}{ \beta_j}}, \quad u\to \IF
% &=\frac{1}{\sqrt{1-\sigma_{ij}(u)^2}} \left(\frac u {\lambda_j}\right)^{\frac{\beta_i \sigma_{ij}(u)}{\beta_j}}\left(1+ \int_{0}^\infty \left(e^{\beta_i\gamma(u)\sqrt{1-\rho(u)^2}s}-e^{-\beta_i\gamma(u)\sqrt{1-\rho(u)^2}s} \right) \frac{\int_{\sqrt{\rho^2+u^2}}^\infty (r^2-u^2)^{-\frac 12} d F_2(r)} {2 \int_{u}^\infty (r^2-u^2)^{-\frac 12} d F_2(r)} ds\right)\\
%&\text{ Still to be calculated or found by a reference}
% &=\frac{1}{\sqrt{1-\sigma_{ij}(u)^2}} \left(\frac u {\lambda_j}\right)^{\frac{\beta \sigma_{ij}(u)}{\beta_j}}\Bigg(1+\\&\quad  \int_{0}^\infty \left(e^{\beta_i\gamma(u)\sqrt{1-\rho(u)^2}\rho}-e^{-\beta_i\gamma(u)\sqrt{1-\rho(u)^2}\rho} \right) \frac{\int_{\sqrt{\rho^2+u^2}}^\infty (r^2-u^2)^{-\frac 12} \int_r^\infty rs^{n-2}(s^2-r^2)^{\frac n2-2} d F(s) dr} {2 \int_{u}^\infty (r^2-u^2)^{-\frac 12}  \int_r^\infty rs^{n-2}(s^2-r^2)^{\frac n2-2} d F(s) dr} d\rho\Bigg)
% &\sim \frac{\Gamma\left(\frac{d}2 \right)}{\pi \Gamma\left(\frac{d-2}2 \right)}\int_{0}^1 \int_{-1}^1 \lambda_i \left(\frac{u}{\lambda_i}\right) ^{\frac{\beta_i \gamma(u) (\theta\sigma_{ij}(u) +\sqrt{1-\theta^2 }\sqrt{1-\sigma_{ij}(u)^2 } x)}{\theta \beta_j \gamma_u } }  (1-x^2)^{\frac{d-4}2}  (1-\theta^2)^{\frac{d-3}2}dx d\theta\\
% &=\frac{\Gamma\left(\frac{d}2 \right)}{\pi \Gamma\left(\frac{d-2}2 \right)} \left(\frac{u}{\lambda_i}\right) ^{\frac{\beta_i \sigma_{ij}(u) }{\beta_j  } }
% \int_{0}^1 \int_{-1}^1 \lambda_i \left(\frac{u}{\lambda_i}\right) ^{\frac{\beta_i \sqrt{1-\theta^2 }\sqrt{1-\sigma_{ij}(u)^2 } x}{\theta \beta_j  } }  (1-x^2)^{\frac{d-4}2}  (1-\theta^2)^{\frac{d-3}2}dx d\theta
\end{align*}
%as $u \to \IF$,
and thus \eqref{secondorder} is in accordance with the findings of Kortschak (2011).
\COM{
b) In view of Davis-Resnick tail property (Proposition 1.1 in Davis and
Resnick (1988), see also Embrechts et al.\ (1997) p.\ 586)) % or Lemma ?? in ???)
if \eqref{eq:rdfd:2} holds and the scaling function $e(\cdot)$ satisfies
\eqref{coEU} holds, then
\BQN\label{Davis}
\limit{ u}  \fracl{u }{e(u)}^\mu  \frac{\overline{F}( \kappa u)}{\overline{F}( u)}&=& 0
\EQN
is satisfied for any $\kappa>1$ and $\mu \inr$.  Consequently, $\Upsilon(u)$ in \eqref{secondorder} tends to 0 as $u \to \IF$.
Note in passing that
\BQN\label{eq:w}
\limit{u} \frac{ e(u)}{u}&=& 0,
\EQN
hence, for $\mu \le 0$ \eqref{Davis} holds since $1- H$ is a rapidly varying function at infinity (see Resnick (1987)).
%%Note in passing that when \eqref{GG} holds, then, as $u\to \IF$
}

\vE{b) The von Mises condition \eqref{eq:fe} is satisfied by a large class of distribution functions in the Gumbel MDA. In particular,
the log-normal distribution satisfies it. Indeed, we have that with $\widetilde f(x)= \frac{1}{ x \sqrt{2 \pi}} \exp(- (\log x)^2/2), x>0$
the scaling function can be taken to be  $e(u)= u/\log u, u>0$, and thus
$$\frac{\widetilde f(u+ x e(u))}{f(u)} \sim  \exp(- (\log (u+ x u/\log u))^2/2 + (\log u)^2/2) \sim \exp(-x), \quad \forall x\inr$$
as $u\to \IF$.}
\end{remark}

{\bf Example}. Given the huge interest on multivariate log-normal models for aggregated \rE{risk} (see e.g., Asmussen et al.\ (2011))
we discuss briefly the findings of our main result when $\log X_i(u), i\le d, u>0$ in \netheo{theorem:secondorder} are Normal random variables. Since the random radius $R$ pertaining to the stochastic representation of the Gaussian distribution and the distribution function of each $\log X_i(u)$ \rE{are} in the Gumbel MDA with the same scaling function $b(u)=1/u$, we have that $e(u)= u/\log u$, hence
$$ \lim_{\lambda  \to 1}\limsup_{u\to \IF} \frac{ e(\lambda u)}{e(u)}=
\lim_{\lambda  \to 1}\limsup_{u\to \IF} \frac{  \lambda u /(\log u+ \log \lambda)}{u/\log u} =\lim_{\lambda  \to 1} \lambda=1.$$
Consequently, condition \eqref{coEU} is satisfied, and further \eqref{eq:fe} follows easily.
Next, for any $j\le d$ %we have for all $u>0$
 \BQNY
 e_j^*(u)&=&  \beta_j\gamma_u u e\left( \luj^ {\frac 1 {\beta_j\gamma_u}}\right)\luj^ {-\frac 1 {\beta_j\gamma_u}} = \frac{\beta_j\gamma_u u}{\log \left( \luj^ {\frac 1 {\beta_j\gamma_u}}\right)}=
 \frac{\beta_j^2\gamma_u^2 u}{\log u - \log \lambda_j} \sim   \frac{\beta_j^2\gamma^2 u}{\log u }
\EQNY
\rE{as $u\to \IF$}. Further, since
\BQNY
c_i=\lim_{u\to\infty} \frac{\log(u)e^*_i(u)}{u} = \gamma^2\beta_i^2, \quad 1 \le i \le d
\EQNY
condition \eqref{condc} holds. Therefore, \eqref{secondorder} boils down to
\BQNY\Upsilon(u)
&
\sim &\frac{\log u }{ \gamma^3 u}  \sum_{j=1}^d \sum_{i\not=j}\frac{\lambda_i}{\cD{\beta_j^3}}
\exp\left( \frac{ c_{\cD{j}}\left( 1-\sigma_{ij}^2( \cE{u})\right)} {  2 \cE{ (\beta_j/\beta_i)}^{2 } }\right) \left(\frac{u}{\lambda_j}\right)^{\frac{\beta_i \sigma_{ij}(u)}{ \beta_j}} \PP\left( X_j(u)>u\right)\\
&
\sim &\frac{\log u }{ \gamma^3 u} \sum_{j=1}^d \sum_{i\not=j}\frac{\lambda_i}{\cD{\beta_j^3}}
\exp\left( \frac{ c_{\cD{j}}\left( 1-\sigma_{ij}^2( \cE{u})\right)} {  2 \cE{ (\beta_j/\beta_i)}^{2 } }\right) \left(\frac{u}{\lambda_j}\right)^{\frac{\beta_i \sigma_{ij}(u)}{ \beta_j}}  \frac{\exp(- (\log (u/\lambda_j))^2/(2\beta_j ^2\gamma_u^2))}{\sqrt{ 2  \pi \beta_j ^2\gamma_u^2} \log u}\\
&
\sim &  \sum_{j=1}^d \sum_{i\not=j}\frac{\lambda_i}{\cD{(\beta_j \gamma)^4}}
\exp\left(  (\beta_i \gamma)^2( 1-\sigma_{ij}^2( \cE{u}))/2\right) \left(\frac{u}{\lambda_j}\right)^{\frac{\beta_i \sigma_{ij}(u)}{ \beta_j}}  \frac{\exp(- (\log (u/\lambda_j))^2/(2 (\beta_j \gamma_u)^2))}{ u\sqrt{ 2  \pi} }, \quad  u\to \IF.
\EQNY
% In the special case of equi-correlated distribution, i.e., $\sigma_{ij}(u)= \rho(u) \to \rho <1, i,j\le d, i\not=j$ we have further
% \BQNY
% \Upsilon(u) &
% \sim & \frac{\exp(- (\log u)^2/2)}{\sqrt{ 2  \pi} \gamma^3 u^{1- \rho(u)} } \sum_{j=1}^d \sum_{i\not=j}\frac{\lambda_i}{\cD{\beta_j^3}}
% \exp\left( \gamma^2 \beta_i^2( 1-\rho^2)/2\right) \left(\frac{1}{\lambda_j}\right)^{\frac{\beta_i \rho}{ \beta_j}}, \quad  u\to \IF.
% \EQNY

\COM{

 The scaling function $e(\cdot)$ is self-neglecting (see Bingham et al.\ (1987) for the main properties), i.e.,
\BQNY
\frac{e(u+ x e(u))}{e(u)} \to 1, \quad u\to \IF
\EQNY
locally uniformly for $x\inr$, and
\BQN\label{eq:w}
\limit{u} \frac{ e(u)}{u}&=& 0.
\EQN
Moreover, \eqref{eq:rdfd:2} implies that for all large $u$ (e.g.,  Resn\cD{ic}k (1987))%
\BQN\label{rep}
1- H(u)& =& c(u) \exp \Bigl( - \int_{x_0}^u \frac{g(t)}{e(t)}\, dt \Bigr), \quad \forall u> x_0,
\EQN
with $c,g$ two positive measurable functions such that $\limit{u}c(u)=\limit{u}g(u)=1$ and $x_0$ some constant.
}

\section{Numerical Examples}\label{section:numerics}

In this section we \qE{shall present} some numerical examples. For comparison purposes we use the same examples as in Mitra and Resnick (2009),
which means  that  $(Y_1,Y_2)$ is a bivariate Normal random vector with zero mean, each component has variance $1$ and correlation coefficient is $\vE{\rho(u)}=:\rho \in \{-0.9,0,0.5,0.9\}$, \qE{and we set $X_i= \exp(Y_i), i=1,2.$} For this choice $e_j^*(u)=e(u)$ for any $u>0, j=1,2$.
For the practical implementation of the second order \EH{asymptotics} we replace the term $\frac{\PP\left( \Zju>u\right)}{e^*_j(u)}$ in equation \eqref{secondorder} by the probability density function of $\Zju$. To check the accuracy of the asymptotic approximations we use Monte Carlo simulation with the estimator from Kortschak and Hashorva (2013). \rE{In tables} \ref{table:09}--\ref{table:m09} we \rE{present} the results of the numerical study. In the first column of the tables we provide the threshold $u$. In the column ''Asympt $1$`` respectively ''Asympt $2$`` we provide the first respectively second order asymptotic approximation for the ruin probability. The results of the Monte Carlo simulation is given in the column ''MC``. For the Monte Carlo simulation we used so many simulations that the error of the Monte Carlo simulation is negligible compared to the error of the asymptotic approximations. The column ''Ratio $1$`` respectively ''Ratio $2$`` provide the ratio of the first order respectively second order asymptotic approximation and the result of the Monte Carlo simulation. The last three columns of the table provide three heuristic measures for the quality of the asymptotic approximations $\epsilon, u^\epsilon$ and $\hat \rho$. \rE{For} the measure $\epsilon$ \rE{which is
motivated by}   condition \eqref{conditionrho} we calculate for $\theta=\log(u)/\log(u+e(u))$ the corresponding $\epsilon$ for which
\BQN
\vE{\rho(u)}+c \sqrt{1-\vE{\rho(u)}^2} \sqrt{1/\theta^2 -1}
= \frac{\beta_{\aH{j}}}{\beta_i} \frac {\log(\epsilon e^*_{\cH{i}}(u))}{\log(u)}.                                                                                                                                                                                                                                                               \EQN
The  measure $\rho$ is just the $\hat \rho$ for which  the inequality in \eqref{remark:conditionrho} is fulfilled with equality.\\
For the case $\rho =0.9$ we can observe that the second order approximation improves significantly over the first order \rE{one} but is still not applicable. \rE{Indeed, in order to} get a  $\hat \rho=0.9$ we have \rE{to} choose  $u\approx 3.4\times 10^{15}$ and similarly if we choose $\rho=\vE{\rho(u)}=0.9$ and we want to get an  $\epsilon\approx 0.1$ we need $u\approx10^{30}$ which means that there should not be any hope that the \EH{asymptotics} gives meaningful results for $\rho=\vE{\rho(u)}=0.9$ and $u$ in a normal range (which was also found
by Mitra and Resnick (2009)). Also for the other values of $\rho$
 we see from the tables that the  second order \EH{asymptotics} improves significantly over the first order estimate. Further the asymptotics work better for smaller values of $\rho$ especially for $\rho=-0.9$ the first as well as the second order \rE{asymptotics} perform quite well. \rE{Finally we remark that the value of $e^\epsilon$ displayed before $\hat \rho$ is comparable with} the value in the column in ''Ratio $1$'' (for sufficiently large values of $u$).
  %So it seems that the measure $\epsilon$ has some relevance in predicting the quality of the first order approximation and should be  further %studied in a different paper.
\begin{table}
\centering
\begin{tabular}{|c||c|c|c||c|c||c|c|c|}
 \hline
$u$ & Asympt $1$ & Asympt $2$ & MC & Ratio $1$ & Ratio $2$ &$\epsilon$&$e^\epsilon$&$\hat\rho$\\\hline\hline
$10$ & $0.0213$ & $0.0705$ & $0.0522$ & $2.45$ & $0.74$ & $2$ & $7.38$ & $0.638$\\\hline
$30$ & $0.000671$ & $0.00259$ & $0.00297$ & $4.43$ & $1.15$ & $2.89$ & $18$ & $0.64$\\\hline
$50$ & $9.15e-05$ & $0.000373$ & $0.000529$ & $5.77$ & $1.42$ & $3.25$ & $25.8$ & $0.651$\\\hline
$75$ & $1.58e-05$ & $6.68e-05$ & $0.000112$ & $7.08$ & $1.67$ & $3.51$ & $33.5$ & $0.661$\\\hline
$100$ & $4.12e-06$ & $1.79e-05$ & $3.36e-05$ & $8.16$ & $1.88$ & $3.68$ & $39.7$ & $0.668$\\\hline
$200$ & $1.17e-07$ & $5.31e-07$ & $1.32e-06$ & $11.3$ & $2.49$ & $4.04$ & $57.1$ & $0.685$\\\hline
$300$ & $1.17e-08$ & $5.45e-08$ & $1.59e-07$ & $13.6$ & $2.91$ & $4.23$ & $68.7$ & $0.695$\\\hline
$500$ & $5.15e-10$ & $2.45e-09$ & $8.68e-09$ & $16.9$ & $3.54$ & $4.43$ & $84.2$ & $0.706$\\\hline
$700$ & $5.71e-11$ & $2.76e-10$ & $1.11e-09$ & $19.4$ & $4.02$ & $4.55$ & $94.7$ & $0.713$\\\hline
$1000$ & $4.92e-12$ & $2.4e-11$ & $1.1e-10$ & $22.3$ & $4.56$ & $4.66$ & $106$ & $0.72$\\\hline
$1500$ & $2.61e-13$ & $1.29e-12$ & $6.78e-12$ & $26$ & $5.27$ & $4.77$ & $118$ & $0.728$\\\hline
$2000$ & $2.94e-14$ & $1.46e-13$ & $8.52e-13$ & $29$ & $5.83$ & $4.84$ & $127$ & $0.733$\\\hline
$2500$ & $5.12e-15$ & $2.56e-14$ & $1.6e-13$ & $31.3$ & $6.26$ & $4.89$ & $133$ & $0.737$\\\hline
$3000$ & $1.18e-15$ & $5.92e-15$ & $3.95e-14$ & $33.4$ & $6.66$ & $4.93$ & $138$ & $0.74$\\\hline
$5000$ & $1.63e-17$ & $8.26e-17$ & $6.47e-16$ & $39.6$ & $7.84$ & $5.02$ & $151$ & $0.748$\\\hline
$7000$ & $8.47e-19$ & $4.29e-18$ & $3.67e-17$ & $43.3$ & $8.55$ & $5.06$ & $158$ & $0.754$\\\hline
$10000$ & $3.25e-20$ & $1.65e-19$ & $1.58e-18$ & $48.7$ & $9.59$ & $5.1$ & $165$ & $0.759$\\\hline
$1e+05$ & $1.14e-30$ & $5.72e-30$ & $9.21e-29$ & $81.1$ & $16.1$ & $5.17$ & $176$ & $0.788$\\\hline
$1e+06$ & $2.05e-43$ & $9.94e-43$ & $2.16e-41$ & $105$ & $21.7$ & $5$ & $148$ & $0.81$\\\hline
\end{tabular}

\caption{ Results of approximation for  $\rho =0.9$}\label{table:09}
\end{table}
\notindraft{
\begin{table}
\centering
\begin{tabular}{|c||c|c|c||c|c||c|c|c|}
 \hline
$u$ & Asympt $1$ & Asympt $2$ & MC & Ratio $1$ & Ratio $2$ &$\epsilon$&$e^\epsilon$&$\hat\rho$\\\hline\hline
$10$ & $0.0213$ & $0.0472$ & $0.0444$ & $2.09$ & $0.941$ & $1.3$ & $3.66$ & $0.638$\\\hline
$30$ & $0.000671$ & $0.00132$ & $0.00165$ & $2.46$ & $1.25$ & $1.2$ & $3.33$ & $0.64$\\\hline
$50$ & $9.15e-05$ & $0.00017$ & $0.000226$ & $2.47$ & $1.33$ & $1.1$ & $3$ & $0.651$\\\hline
$75$ & $1.58e-05$ & $2.78e-05$ & $3.79e-05$ & $2.4$ & $1.36$ & $1.01$ & $2.74$ & $0.661$\\\hline
$100$ & $4.12e-06$ & $7e-06$ & $9.54e-06$ & $2.31$ & $1.36$ & $0.939$ & $2.56$ & $0.668$\\\hline
$200$ & $1.17e-07$ & $1.83e-07$ & $2.41e-07$ & $2.06$ & $1.32$ & $0.779$ & $2.18$ & $0.685$\\\hline
$300$ & $1.17e-08$ & $1.75e-08$ & $2.23e-08$ & $1.9$ & $1.27$ & $0.691$ & $2$ & $0.695$\\\hline
$500$ & $5.15e-10$ & $7.28e-10$ & $8.86e-10$ & $1.72$ & $1.22$ & $0.589$ & $1.8$ & $0.706$\\\hline
$700$ & $5.71e-11$ & $7.82e-11$ & $9.21e-11$ & $1.61$ & $1.18$ & $0.528$ & $1.7$ & $0.713$\\\hline
$1000$ & $4.92e-12$ & $6.52e-12$ & $7.48e-12$ & $1.52$ & $1.15$ & $0.468$ & $1.6$ & $0.72$\\\hline
$1500$ & $2.61e-13$ & $3.34e-13$ & $3.68e-13$ & $1.41$ & $1.1$ & $0.407$ & $1.5$ & $0.728$\\\hline
$2000$ & $2.94e-14$ & $3.68e-14$ & $4e-14$ & $1.36$ & $1.09$ & $0.368$ & $1.44$ & $0.733$\\\hline
$2500$ & $5.12e-15$ & $6.3e-15$ & $6.74e-15$ & $1.32$ & $1.07$ & $0.339$ & $1.4$ & $0.737$\\\hline
$3000$ & $1.18e-15$ & $1.44e-15$ & $1.52e-15$ & $1.29$ & $1.06$ & $0.318$ & $1.37$ & $0.74$\\\hline
$5000$ & $1.63e-17$ & $1.92e-17$ & $2e-17$ & $1.22$ & $1.04$ & $0.263$ & $1.3$ & $0.748$\\\hline
\end{tabular}

\caption{ Results of approximation for $\rho =0.5$}\label{table:05}
\end{table}

\begin{table}
\centering
\begin{tabular}{|c||c|c|c||c|c||c|c|c|}
 \hline
$u$ & Asympt $1$ & Asympt $2$ & MC & Ratio $1$ & Ratio $2$ &$\epsilon$&$e^\epsilon$&$\hat\rho$\\\hline\hline
$10$ & $0.0213$ & $0.0306$ & $0.0337$ & $1.58$ & $1.1$ & $0.537$ & $1.71$ & $0.638$\\\hline
$30$ & $0.000671$ & $0.000806$ & $0.000864$ & $1.29$ & $1.07$ & $0.28$ & $1.32$ & $0.64$\\\hline
$50$ & $9.15e-05$ & $0.000104$ & $0.000108$ & $1.18$ & $1.04$ & $0.196$ & $1.22$ & $0.651$\\\hline
$75$ & $1.58e-05$ & $1.74e-05$ & $1.77e-05$ & $1.12$ & $1.02$ & $0.145$ & $1.16$ & $0.661$\\\hline
$100$ & $4.12e-06$ & $4.45e-06$ & $4.5e-06$ & $1.09$ & $1.01$ & $0.117$ & $1.12$ & $0.668$\\\hline
$200$ & $1.17e-07$ & $1.22e-07$ & $1.23e-07$ & $1.05$ & $1$ & $0.0679$ & $1.07$ & $0.685$\\\hline
$300$ & $1.17e-08$ & $1.21e-08$ & $1.21e-08$ & $1.03$ & $1$ & $0.049$ & $1.05$ & $0.695$\\\hline
$500$ & $5.15e-10$ & $5.25e-10$ & $5.26e-10$ & $1.02$ & $1$ & $0.0322$ & $1.03$ & $0.706$\\\hline
$700$ & $5.71e-11$ & $5.8e-11$ & $5.8e-11$ & $1.02$ & $1$ & $0.0243$ & $1.02$ & $0.713$\\\hline
$1000$ & $4.92e-12$ & $4.98e-12$ & $4.98e-12$ & $1.01$ & $1$ & $0.018$ & $1.02$ & $0.72$\\\hline
\end{tabular}

\caption{ Results of approximation for  $\rho =0$}\label{table:00}
\end{table}
\begin{table}
\centering
\begin{tabular}{|c||c|c|c||c|c||c|c|c|}
 \hline
$u$ & Asympt $1$ & Asympt $2$ & MC & Ratio $1$ & Ratio $2$ &$\epsilon$&$e^\epsilon$&$\hat\rho$\\\hline\hline
$2$ & $0.488$ & $0.673$ & $0.785$ & $1.61$ & $1.17$ & $0.331$ & $1.39$ & $1.53$\\\hline
$3$ & $0.272$ & $0.331$ & $0.369$ & $1.36$ & $1.11$ & $0.263$ & $1.3$ & $0.914$\\\hline
$5$ & $0.108$ & $0.119$ & $0.121$ & $1.12$ & $1.02$ & $0.148$ & $1.16$ & $0.704$\\\hline
$10$ & $0.0213$ & $0.0221$ & $0.0221$ & $1.04$ & $1$ & $0.0555$ & $1.06$ & $0.638$\\\hline
$15$ & $0.00677$ & $0.0069$ & $0.0069$ & $1.02$ & $1$ & $0.0297$ & $1.03$ & $0.632$\\\hline
$30$ & $0.000671$ & $0.000675$ & $0.000675$ & $1.01$ & $1$ & $0.00976$ & $1.01$ & $0.64$\\\hline
$50$ & $9.15e-05$ & $9.18e-05$ & $9.18e-05$ & $1$ & $1$ & $0.00419$ & $1$ & $0.651$\\\hline
$75$ & $1.58e-05$ & $1.58e-05$ & $1.58e-05$ & $1$ & $1$ & $0.00212$ & $1$ & $0.661$\\\hline
$100$ & $4.12e-06$ & $4.12e-06$ & $4.12e-06$ & $1$ & $1$ & $0.0013$ & $1$ & $0.668$\\\hline
\end{tabular}

\caption{Results of approximation for  $\rho =-0.9$}\label{table:m09} %(For small values of $u$ we used the CMC method since there my implementation of the estimator is not correct)}
\end{table}

% \begin{table}
% \centering
% \begin{tabular}{|c||c|c|c||c|c||c|c|}
%  \hline
% $u$ & Asympt $1$ & Asympt $2$ & MC & Ratio $1$ & Ratio $2$ & HW $1$ & HW $2$\\\hline\hline
% 10 & $0.021$ & $0.07$ & $0.051$ & $2.4$ & $0.73$ & $0.048$ & $0.048$\\\hline
% 30 & $0.00067$ & $0.0026$ & $0.0029$ & $4.4$ & $1.1$ & $0.15$ & $0.15$\\\hline
% 50 & $9.2e-05$ & $0.00037$ & $0.00052$ & $5.7$ & $1.4$ & $0.25$ & $0.25$\\\hline
% 75 & $1.6e-05$ & $6.7e-05$ & $0.00011$ & $7.1$ & $1.7$ & $0.36$ & $0.36$\\\hline
% 100 & $4.1e-06$ & $1.8e-05$ & $3.3e-05$ & $8$ & $1.8$ & $0.45$ & $0.45$\\\hline
% 200 & $1.2e-07$ & $5.3e-07$ & $1.3e-06$ & $11$ & $2.4$ & $0.79$ & $0.79$\\\hline
% 300 & $1.2e-08$ & $5.4e-08$ & $1.7e-07$ & $15$ & $3.1$ & $1.2$ & $1.2$\\\hline
% 500 & $5.1e-10$ & $2.4e-09$ & $8.7e-09$ & $17$ & $3.6$ & $1.6$ & $1.6$\\\hline
% 700 & $5.7e-11$ & $2.8e-10$ & $1.2e-09$ & $21$ & $4.3$ & $2.1$ & $2.1$\\\hline
% 1000 & $4.9e-12$ & $2.4e-11$ & $Inf$ & $Inf$ & $Inf$ & $NaN$ & $NaN$\\\hline
% \end{tabular}
%
% \caption{ Results of approximation for $\rho =0.9$ other estimator}
% \end{table}
}
\section{Further  Results and Proof of \netheo{theorem:secondorder}} \label{section:proofs}
We present first four lemmas which are of some independent interest and then proceed with the proof of our main result.
In the sequel we consider some positive random variable  $R$  such that its distribution function  $F$
has an infinite upper endpoint. Under the assumption that $F$ is in Gumbel MDA with some positive scaling function $b(\cdot)$, we have the following representation (see e.g., Resn\cD{ic}k (1987))
\BQN\label{rep}
1- \qE{F}(u)& =& c(u) \exp \Bigl( - \int_{x_0}^u \frac{g(t)}{b(t)}\, dt \Bigr),
\EQN
with $x_0$ some constant and $c(\cdot),g(\cdot)$ two positive measurable functions such that $\limit{u}c(u)=\limit{u}g(u)=1$.

\pE{Below we} assume that $e(u)= u b(\log(u))$ is a scaling function of $\widetilde F$, i.e.,
\rE{the df} $\widetilde F$ is in the Gumbel MDA with scaling function $e(\cdot)$
(recall $\widetilde F$ is the distribution function of $\exp(R)$). This holds in particular when $\limit{u} b(u)\cH{=0}$.\\
Next,   define $e^*(u)$ by \eqref{eq:eyj} for some $\lambda, \beta$ positive, \hH{i.e.,}
\hH{$$
 e^*(u)= \beta \gamma_u u e\left( \luiu^ {\frac 1 {\beta\gamma_u}}\right)\luiu^ {-\frac 1 {\beta\gamma_u}},
 $$
with $\gamma_u$ such that $\limit{u} \gamma_u= \gamma\in (0,\IF)$.
}

The next lemma is shown in Kortschak and Hashorva (2013), whereas \nelem{remarkunivariateasymptotic} follows by  Berman (1992), see also
Hashorva (2012).  Let next $\vk{v}$ be a given vector in $\R^d,d\ge2$ with $L_2-$norm equal 1, and define $\vk{\theta}= A_u \vk{v}$.

\begin{lem}\label{cor:last} Under Assumption \eqref{conditionrho}, for every $j$ with $\beta_j=\beta_1$ and every
\aK{$\epsilon>0$} there exist \uE{some $c,u_0$} positive such that
  $$
 \theta_i\le \theta_j \frac{\beta_j} {\beta_i} \frac{\log(\epsilon e^*_j(u))}{\log(u)}
 $$
 holds for all $u>u_0$, provided that $\theta_j>1-c/\log(u)$.
\end{lem}

\begin{lem}\label{remarkunivariateasymptotic}
Let $R$ be a positive random variable, and let $\ftj$ be given by
\BQN\label{dtheta}
\ftj(x)&=&\frac{\Gamma(d/2)}{\sqrt{\pi} \Gamma((d-1)/2)} (1-x^2)^{\frac{d-3}2}, \quad x\in  (0,1),
\EQN
\uE{with $\Gamma(\cdot)$ the Euler's Gamma function.} \hH{If $F$ is in the Gumbel MDA with some positive scaling function $e(\cdot)$}, then for any
\hH{$\beta,\lambda$ positive}
 \begin{equation}
\PP(X_j(u)>u)=\int_{0}^1\PP\left( \lambda e^{R \theta \beta \gamma_ {u}}>u\right) \ftj(\theta) d \theta\\
\sim  \frac{2^{\cD{\frac{d-\cK{3}}2}}\Gamma(d/2)}{\sqrt{\pi}}\left(\frac{e^*(u)}{u\log(u)}\right)^{\frac{d-1} 2}\PP\left( \lambda e^{R \beta \gamma_ {u}}>u \right), \quad u\to \IF.
 \end{equation}
\end{lem}

\def\vuu{\bH{\xi(u)}}

In the sequel for two  positive functions $g_1, g_2$ we write  \cK{$g_1 \lesssim g_2$ respectively $g_1 \gtrsim g_2$ if $\limsup_{u\to \IF} g_1(u)/g_2(u)\le 1 $ respectively $\limsup_{u\to \IF} g_1(u)/g_2(u)\ge 1 $.}

\cD{
\begin{lem}\label{lemmaboundtaillog} Under the assumptions of \nelem{remarkunivariateasymptotic}, if further
\BQN
\limit{u} \frac{ \rE{e(u)}\log(u) }{u}&=& c_0
\label{fred}
\EQN
 holds for some constant $c_0\in [0,\IF)$, then
for any $c>1,c'>c_0$
 \begin{equation*}
\frac{\PP\left(R >\log(cu)\right)}{ \PP\left( R >\log(u)\right)}\lesssim u^{-\frac{\log(c)}{c'}}
 \end{equation*}
is valid \qE{for all $u$ large}.
 \end{lem}
}
\begin{proof} \cD{
By \eqref{fred} and the representation \eqref{rep} of the scaling function $e(\cdot)$ we obtain for $c'>c_0$
and $\ve>0, 1/u$ sufficiently small}
\begin{align*}
\frac{\PP\left(\cE{R > \log(cu)}\right)}{\PP\left(R > \log(u)\right)} &\lesssim  \exp\left(-\cE{(1-\ve)}\int_{u}^{cu} \frac{1}{e(x)} dx\right)\\
&\lesssim  \exp\left(- \frac{1}{c'} \int_{u}^{cu} \frac{\log(x)}{x} dx\right)\\
&=\exp\left(- \frac{1}{c'} \int_{1}^{c} \frac{\log(u x)}{x} dx\right)\\
&\le\exp\left(- \frac{1}{c'} \int_{1}^{c} \frac{\log(u )}{x} dx\right)%\\
%&=u^{-\frac{\log(c)}{c'}}.
\end{align*}
thus the claim follows.
\end{proof}

\EH{
Next we shall consider the case that $e(\cdot)$  is O-regularly varying which means (see e.g., Bingham et al.\ (1987))
\BQN\label{Oreg}
0 < \liminf_{u \to \IF} \frac{ e(\lambda u)}{e(u)}\le \limsup_{u \to \IF} \frac{ e(\lambda u)}{e(u)} < \IF, \quad
\forall \lambda>1.
\EQN
}
\begin{lem}\label{lemma:expbound}
If  $e(\cdot)$  satisfies \eqref{Oreg}, then there exist $\alpha, M, u_0$ positive and \rE{$\ve\in (0,1)$ such that}
\BQNY
\frac{\PP(R > \log( u+xe(u)))}{\PP(R> \log (u))} %I_{\{x<c u/e(u)\}}
&\le  &\cE{(1+\ve)} e^{-\cE{(1-\ve)}\cD{M}(1+c)^{-\alpha} x}
\EQNY
holds for any $c>0$  and  $x \in (0,\frac{c u}{e(u)})$ with $u>u_0$.
\end{lem}
\begin{proof} By Proposition 2.2.1 of Bingham et al. (1987) there exists $M$, $u_0$ and $\alpha$ such that for all $y\ge x\ge u_0$
$$
\frac{e(x)}{e(y)}\ge M (x/y)^\alpha.
$$
By the representation \eqref{rep} of the scaling function $e(\cdot)$  and the assumptions, for \qE{$x \in (0, c e(u)/u)$}
\rE{for some $\ve \in (0,1)$} we may write
\begin{align*}
\frac{\PP(e^R>u+xe(u))}{\PP(e^R>u)} %I_{\{x<c u/e(u)\}}
& \le \cE{(1+\ve)}\exp\left(\cE{-(1- \ve)}\int_{u}^{u+xe(u)} \frac{1}{e(y)}\, dy \right)\\%I_{\{x<c u/e(u)\}}\\
& = \cE{(1+\ve)} \exp\left(\cE{-(1- \ve)} \int_{\cE{0}}^{x} \frac{e(u)}{e(u+ye(u))}\, dy \right)\\%I_{\{x<c u/e(u)\}}\\
& \le \cE{(1+\ve)} \exp\left(\cE{-(1- \ve)} M \int_{\cE{0}}^{x} \left(\frac{u}{u+ye(u)}\right)^{\alpha} dy \right)\\%\le (1+\c_1 e^{-c^{-\alpha}x}
&\le{(1+\ve)} \exp\left({-(1- \ve)} M \int_{\cE{0}}^{x} \left(1+c\right)^{-\alpha} dy \right),
\end{align*}
hence the proof follows. \end{proof}

%%%%%%%%%%%%%%%%%%%%%%%%%%%%%%%%%%%%%%%%%%%%%%%%%%%%
%%%%%%%%%%%%%%%%%%%%%%%%%%%%%%%%
\prooftheo{theorem:secondorder}
For all $u$ positive we have
\[
\Upsilon(u)= \PP(S(u)>u)- \sum_{j=1}^d \PP\left(X_j(u)>u\right) =\sum_{j=1}^d \Biggl(\PP\left(S(u)>u,X_j(u)>\max_{ j\not=i} X_j(u) \right)- \PP\left(X_j(u)>u\right) \Biggr).
\]
For $j$ with $\beta_j<\beta_1$ \rE{we have further}
\[
\PP\left(S(u)>u,X_j(u)>\max_{ j\not=i} X_i(u) \right)\le \PP\left(X_j(u) >u/d\right)=\PP\left(X_1(u) > \lambda_1\left(\frac{u}{\lambda_i d} \right)^{\beta_1/\beta_j}\right)=o\left(\Upsilon(u)\right)
\]
as $u\to \IF$.
Hence we only \rE{need to \rE{derive} the second order asymptotics of}
\[
 \PP\left(S(u)>u,X_j(u)>\max_{ j\not=i} X_i(u) \right),
\]
  with $\beta_j=\beta_1$.
Note that by straightforward arguments it follows from  \eqref{coEU} and the monotonicity of $e(\cdot)$ that the scaling function $e(\cdot)$ is O-regularly varying.  Define $ \vk{\Theta}:=A_u \vk{U}$ and write ${\Theta_i}$ for the $i$th component of ${\vk{\Theta}}_u$. Choose an index  $j$ with $\beta_1=\beta_j$, and
suppose without loss of generality that depending on $j$ an $A_u$ is chosen such that  $\bH{\Theta_j}=U_j$, with $U_j$ the $j$th component of
$\U$ which is uniformly distributed on the unit sphere of $\R^d$.   Lemma \ref{lemmaboundtaillog} implies
that we can choose a $k$ such that for
$$a(u)=1-\frac{k}{\log(u)}$$
we have
\begin{equation}
\limit u \frac{ \PP\left(e^{R  \beta_j \gamma_u}>\left(\frac u{\lambda_j d}\right)^{1/a(u)}\right)}{u^{-3- \frac{d-1}2} \PP\left(\lambda_j e^{R  \beta_j \gamma_u}>u\right) }=0.
\label{equsecond1}
\end{equation}
Note that it follows from \eqref{equsecond1} and Lemma \ref{remarkunivariateasymptotic} that for all $u$ large
\begin{align*}
 \PP\left(S(u)>u ,X_j(u)=\max_{i\not=j}X_i(u),\Theta_j\le a(u)\right) &\le \PP\left(X_j(u)>u/d,\Theta_j\le a(u)\right)\\&\le \PP\left( e^{R  \beta_j \gamma_u}>\left(\frac u{\lambda_j d}\right)^{1/a(u)}\right)\\
&=o\left(u^{-3- \frac{d-1}2} \PP\left(\lambda_j e^{R  \beta_j \gamma_u}>u\right)\right)\\
&=o\left(u^{-2} \PP\left(X_j(u)>u\right)\right)=o\left(\Upsilon(u)\right).
\end{align*}
\pE{Therefore, we need to determine} the second order asymptotics of $\PP\left(S(u)>u ,X_j(u)=\max_{i\not=j}X_i(u),\Theta_j> a(u)\right)$.
Denote by $f_{-j}(\vk{\theta}_{-j}|\theta)$ the conditional \rE{probability density function} of $\vk{\Theta}_{-j}:=(\Theta_1,\ldots,\Theta_{j-1},\Theta_{j+1},\ldots,\Theta_d)$ given $\Theta_j=\theta$.  We get by
\eqref{conditionrho} (compare Lemma \ref{cor:last}) that for sufficiently large $u$ with $\ftj$ given by \eqref{dtheta}
\begin{align*}
& \int_{a(u)}^1  \PP\left(\left.\sum_{i=1}^d \lambda_i e^{R \Theta_i  \beta_i \gamma_u}>u,\lambda_j  e^{R \Theta_j\beta_j \gamma_u}>\max_{k\not=j}\lambda_k  e^{R \Theta_k\beta_k \gamma_u}\right|\Theta_j=\theta_j \right) \ftj(\theta_j) d \theta_j\\
&\cD{=}\int_{a(u)}^1\PP\left(\left.\sum_{i=1}^d \lambda_i e^{R \Theta_i \beta_i \gamma_u}>u\right|\Theta_j=\theta \right) \ftj(\theta) d \theta\\
& =\int_{a(u)}^1\int\PP\left(\left.\sum_{i=1}^d \lambda_i e^{R \Theta_i  \beta_i \gamma_u}>u\right|\Theta_1=\theta_1,\ldots,\Theta_d= \theta_d \right) \ftj(\theta_j)f_{\qE{\vk{\theta}_{-j}}}(\qE{\vk{\theta}_{-j}},\theta_j) d \qE{\vk{\theta}_{-j}} d \theta_j.
\end{align*}
\def\ttH{\uE{\theta}}
 For sufficiently large $u$ define the function
 $$g(u)=\left\{e^{r(u)}, \uE{\text{where $r(u)$ is such that }} \sum_{i=1}^d \lambda_i e^{r(u) \ttH_i  \beta_i \gamma_u}=u\right\}.$$
\uE{Hence, for all $u$ large}
$$
\PP\left(\left.\sum_{i=1}^d \lambda_i e^{R \Theta_i \beta_i \gamma_u}>u\right|\Theta_1=\theta_1,\ldots,\Theta_d= \theta_d  \right)=\PP\left(e^R>g(u)\right)=\PP\left(e^R>g_0(u)+ g_1(u)\right),
$$
where
$$
g_0(u)=\left(\frac{u}{\lambda_j}\right)^{\frac 1{\ttH_j \beta_j\gamma_u}}, \quad g_1(u)=g(u)+g_0(u), \quad u>0,
$$
with $g_1(u)<0$ for all $u$ positive. From Lemma \ref{cor:last} we get that uniformly in $\vk{\theta}=\qE{(\theta_1 \ldot \theta_d)}$ with $\theta_j>a(u)$ $\limit{u}g_1(u)/g_0(u)=0$ and hence for some $|\xi|\le g_1(u)$ the following equalities are equivalent
\begin{align*}
 u&=\sum_{i=1}^d \lambda_i (g_0(u)+g_1(u))^{ \ttH_i \beta_i \gamma_u}\\
 u&= \lambda_i (g_0(u))^{ \ttH_j \beta_j \gamma_u} +\ttH_j \beta_j \gamma_u\lambda_i g_1(u) (g_0(u)+\xi)^{ \ttH_j \beta_j \gamma_u-1}+\sum_{i\not= j}^d \lambda_i (g_0(u)+g_1(u))^{ \ttH_i \beta_i \gamma_u}\\
-\ttH_j g_1(u) &=\left(\frac{g_0(u)+g_1(u)}{g_0(u)+\xi}\right)^{ \ttH_j \beta_j \gamma_u-1} \sum_{i\not= j}^d \frac{\lambda_i}{\beta_j\gamma\lambda_j} (g_0(u)+g_1(u))^{\gamma_u( \ttH_i \beta_i -\ttH_i \beta_i)+1}.
\end{align*}
\rE{It} follows that for some  $c>0$ and uniformly in $\vk{\theta}$ with $\theta_j>a(u)$
\begin{align*}
g_1(u)&\sim -\sum_{i\not=j}\frac{\lambda_i}{\cD{\beta_j\gamma}\lambda_j} g_0(u)^{\gamma_u(\ttH_i\beta_i-\ttH_j\beta_j)+1}\\
&\gtrsim  -g_0(u)^{-\gamma_u\ttH_j\beta_j+1} \sum_{i\not=j} \frac{\lambda_i}{\cD{\beta_j\gamma}\lambda_j}g_0(u)^{\gamma_u \beta_j \ttH_{\cD{j}}\frac {\log(\epsilon e^*_j(u))}{\log(u)}}\gtrsim -cg_0(u)^{-\gamma_u\ttH_j\beta_j+1}  \epsilon e^*_j(u).
\end{align*}
Since the scaling function $e(\cdot)$ is O-regularly varying   we get for some $c_1>0$ that $ |g_1(u)| \lesssim c_1 \epsilon e(g_0(u))$ for any  $\epsilon>0$ and uniformly in $\vk{\theta}$ with $\theta_j>a(u)$.
Taylor expansion implies for a $g_0(u)+g_1(u)\le \xi_u\le g_0(u) $
\begin{align*}
\PP\left(e^R>g(u)\right)&=\PP\left(e^R>\left(\frac{u}{\lambda_j}\right)^{\frac 1{\Theta_j \beta_j\gamma_u}}\right)\cD{-}g_1(u)\cE{\widetilde{f}}(\xi_u)\\
&=\PP\left(e^R>\left(\frac{u}{\lambda_j}\right)^{\frac 1{\Theta_j \beta_j\gamma_u}}\right)- (1+o(1))g_1(u)\frac1{e(\xi_u)}\PP(R>\log(\xi_u)).
\end{align*}
Next, note  that with \eqref{equsecond1} and Lemma \ref{remarkunivariateasymptotic}  we get for {$u$ large enough such that $\PP(X_j(u)>u,\theta_j\le u)=0$ that}
\begin{align*}
&\int_{a(u)}^1\int\PP\left(\left.e^R>\left(\frac{u}{\lambda_j}\right)^{\frac 1{\Theta_j \beta_j\gamma_u}}\right|\Theta_1=\theta_1,\ldots,
\Theta_d=\theta_d \right) \ftj(\theta_j)f_{\qE{\vk{\theta}_{-j}}}(\qE{\vk{\theta}_{-j}},\theta_j) d \qE{\vk{\theta}_{-j}} d \theta_j\\
&=\PP\left( X_j(u)>u\right) -\int_{0}^{a(u)}\PP\left(\left.\lambda_j e^{R \beta_j\gamma_u}> \lambda_j^{\frac 1{\Theta_j}+1} u^{\frac 1{\Theta_j}}\right|\Theta_j=\theta_j \right) \ftj(\theta_j) d \theta_j\\
&=\PP\left( X_j(u)>u\right)+\cD{o\left(u^{-3- \frac{d-1}2} \PP\left(\lambda_j e^{R  \beta_j \gamma_u}>u\right)\right)}\\
&=\cD{\PP\left( X_j(u)>u\right) +o\left(u^{-2} \PP\left(X_j(u)>u\right)\right)}.
\end{align*}
 We are left with finding  the \EH{asymptotics} of
\begin{align*}
&-\int_{a(u)}^1\int  g_1(u)\frac1{e(\xi_u)}\PP(R>\log(\xi_u))\ftj(\theta_j)f_{\qE{\vk{\theta}_{-j}}}(\qE{\vk{\theta}_{-j}},\theta_j) d \qE{\vk{\theta}_{-j}} d \theta_j\\
&\sim\sum_{i\not=j} \int_{a(u)}^1\int \frac{\lambda_i}{\cD{\beta_j\gamma}\lambda_j} g_0(u)^{\gamma_u(\ttH_i\beta_i-\ttH_j\beta_j)+1}\frac1{e(\xi_u)}\PP(R>\log(\xi_u))\ftj(\theta_j)f_{\qE{\vk{\theta}_{-j}}}(\qE{\vk{\theta}_{-j}},\theta_j) d \qE{\vk{\theta}_{-j}} d \theta_j.
\end{align*}

Since any scaling function, and therefore $e(\cdot)$ is self-neglecting (see Bingham et al.\ (1987) for the main properties), i.e.,
\BQNY
\frac{e(u+ x e(u))}{e(u)} \to 1, \quad u\to \IF
\EQNY
locally uniformly for $x\inr$
 we get that
\BQNY
g_1(u)\frac 1{e(\xi_u)}\PP(e^R>\xi_u)&\sim & g_1(u)\frac1{e(g_0(u))}\PP(e^R>g_0(u))\\
&=&g_1(u)\frac1 {e(g_0(u))}\PP\left(\lambda_je^{\beta_j \gamma_u \Theta_j}>u\right).
\EQNY
It follows that we need further to calculate the \EH{asymptotics} of
 \begin{align*}
\delta(u)&:= \int_{a(u)}^1\int_{-1}^1 \frac{\lambda_i}{\beta_j\gamma\lambda_j} g_0(u)^{\gamma_u(\ttH_i\beta_i-\ttH_j\beta_j)+1}\frac1{e(g_0(u))}\PP(R>\log(g_0(u)))\ftj(\theta_j)f_{ij}(\theta_{i} \lvert \theta_j)  d \qE{\vk{\theta}_{-j}} d \theta_j
\end{align*}
where $f_{ij}(\theta_{i} \lvert \theta) $ is the probability density function of $\Theta_i \lvert \Theta_j=\theta$. To evaluate $\delta(u)$
 we choose for  a specific index $i$  the  matrix $A_u$ in such a way that $a_{jj}=1$, $a_{ji}=\sigma_{ij}(u)$ and $a_{ii}=\sqrt{1-\sigma_{ij}(u)^2}$. If  $U_i$ and $U_j$ are two components of a random vector that is uniformly distributed on the $d$--dimensional unit sphere and  $V_j$ and $V_i$ are independent  and have the same distribution as the marginal distribution of random vector that is uniformly distributed on the $d$ respectively $(d-1)$--dimensional unit sphere, then Cambanis et al. (1981) Lemma 2 shows that $(U_i,U_j)$ can be represented as
\[
 (U_j,U_i) \stackrel d= \left(V_j,V_i\sqrt{1-V_j^2}\right).
\]
 It follows that we can assume that  $\Theta_j=V_j$ and $\Theta_i= \sigma_{ij}(u) V_j +\sqrt{1-\sigma_{ij}(u)^{\qE{2}}} \sqrt{1-V_j^2} V_i$. Define next the function $\widehat e_j(u,v)$ by
$$
\widehat e_j(u,v):=u \beta_j\gamma_{v}e\left(\left(\frac u{\lambda_j}\right)^ {\frac 1 {\beta_j\gamma_{v}}}\right)\left(\frac u{\lambda_j}\right)^ {-\frac 1 {\beta_j\gamma_{v}}}, \quad u,v>0
$$
and note that $\widehat e_j(u,v)=e_j^*(u)$. We have
$$
\frac{g_0(u)^{{-\gamma_u\theta_j\beta_j+1}}} {e(g_0(u))}= \frac{\lambda_j g_0(u)} { u e(g_0(u))}=\lambda_j\left( \frac{u}{\lambda_j}\right)^{1-\theta_j} \frac1{\widehat e_j\left(u^{1/\theta_j}(\lambda_j)^{1-1/\uE{\theta_j}},u\right)}.
$$
It follows that
\begin{align*}
\delta(u)
&= \frac{\Gamma(d/2)}{\pi \Gamma((d-2)/2)} \frac{\lambda_i}{\beta_j\gamma} \left(\frac{u}{\lambda_j}\right)^{\frac{\beta_i \sigma_{ij}(u)}{ \beta_j}}\int_{a(u)}^1\int_{-1}^1\left( \frac{u}{\lambda_j}\right)^{1- v_j} \left(\frac{u}{\lambda_j}\right)^{\frac{\beta_i v_i\sqrt{1-\sigma_{ij}(u)^2} \sqrt{1- v_j^2}}{ v_j \beta_j}} \\
&\quad \times\frac{\PP\left(\lambda_je^{\gamma_u  v_j\beta_jR}>u\right)}{\widehat e_j\left(u^{1/ v_j}(\lambda_j)^{1-1/ v_j},u\right)}(1- v_j^2)^{\frac{d-3}2} (1- v_i^2)^{\frac{d-4}2}d  v_{i}d v_{j}.
\end{align*}
 Substituting  $v_j=\frac{\log(u)}{\log(u)+\log(1+xe^*_j(u)/u) }=\frac{\log(u)}{\log(u+xe^*_j(u)) }$ we \vE{obtain (set next $\eta_j(u):=1+xe^*_j(u)/u$)}
 \begin{align*}
 \delta(u)&=  \frac{\Gamma(d/2)}{\pi \Gamma((d-2)/2)} \frac{\lambda_i}{\beta_j\gamma} \left(\frac{u}{\lambda_j}\right)^{\frac{\beta_i \sigma_{ij}(u)}{ \beta_j}}\\&\quad \times\int_0^{\frac{e^{\log(u)/a(u)} -u}{e_j^*(u)}}\int_{-1}^1 \frac{\log(u)}{(\log(u)+\log(\eta_j(u)))^2}
\frac{e^*_j(u)/u}{\eta_j(u)}\\
&\quad \times \left( \frac{u}{\lambda_j}\right)^{\frac{\log(\eta_j(u))}{\log(u)+\log(\eta_j(u)) }} \left(\frac{u}{\lambda_j}\right)^{\frac{\beta_iv_i\sqrt{1-\sigma_{ij}(u)^2} \sqrt{2 \log(u) \log(\eta_j(u)) +\log(\eta_j(u))^2}}{ \beta_j \log(u)}}\\
&\quad \times
\frac{\PP\left(\lambda_je^{\gamma_u \beta_jR}>u+ xe_j^*(u)\right)}{\widehat e_j\left(\left(u+xe^*_j(u)\right) (\lambda_j)^{-\frac{\log(\eta_j(u)) }{\log(u)}},u\right)}%\\&\times
\left(\frac{2 \log(u) \log(\eta_j(u)) +\log(\eta_j(u))^2}{(\log(u)+\log(\eta_j(u)) )^2}\right)^{\frac{d-3}2} (1-v_i^2)^{\frac{d-4}2}d v_i d x.
\end{align*}

 We remark that by \eqref{coEU}
$$
\limit{u}\frac{\widehat e_j(u,u)}{\widehat e_j\left(\left(u+xe^*_j(u)\right) (\lambda_j)^{-\frac{\log(\eta_j(u)) }{\log(u)}},u\right)}=1
$$
and using Lemma \ref{lemma:expbound} to get an integrable upper bound (note that $x\lesssim e^{k} u/e^*
 _j(u)$),  by \rE{the} bounded convergence theorem, we obtain
\begin{align*}
 \delta(u)&\sim \frac{2\Gamma(d/2)}{\pi \Gamma((d-2)/2)}\frac{\lambda_i}{\cD{\beta_j\gamma}} \left(\frac{u}{\lambda_j}\right)^{\frac{\beta_i \sigma_{ij}(u)}{ \beta_j}}\left(\frac{2e^*_j(u)}{u\log(u)}\right)^{\frac{d-1}2} \frac{\PP\left(\lambda_je^{\gamma_u \beta_jR}>u\right)}{e^*_j\left(u\right)}
\\&\quad\times\int_{0}^\infty \int_{-1}^1 e^{ \sqrt{x}v \frac{\beta_i\sqrt{2c_0\left(1-\sigma_{ij}^2\right)} }{ \beta_j}} e^{-x}x^{\frac{d-3}2} (1-v^2)^{\frac{d-4}2}d v d x, \quad u\to \IF.
\end{align*}
Using Euler's duplication formula
$$ \Gamma(s) \Gamma(s+1/2)=  \sqrt{\pi} 2^{1- 2s} \Gamma(2s)$$
for $n\ge 0$ we have
$$ \Gamma(n+1/2 ) \Gamma(n+1)=  \sqrt{\pi} 2^{1- 2(n+1/2)} \Gamma(2n+1)= \sqrt{\pi} 4^{-n}\Gamma(2n+1),$$
hence for the last integral above we may further write (\rE{below} $I_{\{\cdot\}}$ stands for the indicator function and
$q:= \beta_i\sqrt{2\cE{c_{\cD{j}}}\left(1-\sigma_{ij}^2\right)}/ \beta_j$)
\begin{align*}
&\int_{0}^\infty \int_{-1}^1 e^{ \sqrt{x}v q} e^{-x}x^{\frac{d-3}2} (1-v^2)^{\frac{d-4}2}d v d x\\
&=\sum_{n=0}^\infty \frac{q^n} { n!} \int_{0}^\infty \int_{-1}^ 1 e^{-x}x^{\frac{d+n-3}2}
v^n  (1-v^2)^{\frac{d-4}2} d v d x\\
&=\sum_{n=0}^\infty \frac{q^n} { n!} \frac{\Gamma\left(\frac{n+1}2\right)\Gamma\left(\frac{d-2}2\right)}{\Gamma\left(\frac{d+n-1}2\right)}\Gamma\left(\frac{d+n-1}2\right) I_{\{n=0 \mod 2\}}\\
&=\Gamma\left(\frac{d-2}2\right)\sum_{n=0}^\infty q^{2n}  \frac{\Gamma\left(\frac{2n+1}2\right)}{\Gamma(2n+1)}\\
&=\sqrt{\pi} \Gamma\left(\frac{d-2}2\right) \sum_{n=0}^\infty (q^2/4)^n\frac{1}{\Gamma(n+1)}=\sqrt{\pi} \Gamma\left(\frac{d-2}2\right) \exp\left(q^2/4 \right).
\end{align*}
Consequently, the claim follows by Lemma \ref{remarkunivariateasymptotic}.
\QED

\end{document}